\documentclass[12pt]{amsart}

\makeatletter
\def\l@section{\@tocline{1}{10pt}{1pc}{}{}}
\def\l@subsection{\@tocline{2}{0pt}{1pc}{4.6em}{}}
\def\l@subsubsection{\@tocline{3}{0pt}{1pc}{7.6em}{}}
\renewcommand{\tocsection}[3]{%
  \indentlabel{\@ifnotempty{#2}{\makebox[2.3em][l]{%
    \ignorespaces#1 #2.\hfill}}}\textbf{#3}}
\renewcommand{\tocsubsection}[3]{%
  \indentlabel{\@ifnotempty{#2}{\hspace*{2.3em}\makebox[2.3em][l]{%
    \ignorespaces#1 #2.\hfill}}}#3}
\renewcommand{\tocsubsubsection}[3]{%
  \indentlabel{\@ifnotempty{#2}{\hspace*{4.6em}\makebox[3em][l]{%
    \ignorespaces#1 #2.\hfill}}}#3}
\makeatother

\usepackage{amsmath,amssymb,amscd,amsthm,graphicx,enumerate,tikz-cd}

\usepackage{hyperref}
\hypersetup{breaklinks=true}

\usepackage{enumitem} 
\newlist{condenum}{enumerate}{1} 
\setlist[condenum]{label=\bfseries Condition \arabic*.,  ref=\arabic*, wide}

\numberwithin{equation}{section}
\theoremstyle{plain}



\usepackage[utf8]{inputenc}
\usepackage{amssymb}
\usepackage{amsthm}
\usepackage{mathrsfs}
\usepackage{amsfonts}
\usepackage{amsmath}
\usepackage{mathtools}
\usepackage{makecell}
\usepackage{xcolor}
\usepackage[margin=1.25in]{geometry}

\usepackage{lipsum}
\makeatletter
\def\ps@pprintTitle{%
 \let\@oddhead\@empty
 \let\@evenhead\@empty
 \def\@oddfoot{}%
 \let\@evenfoot\@oddfoot}
\makeatother

\newcommand{\R}{\mathbb{R}}
\newcommand{\RR}{\mathbb{R}^2}

\newcommand{\Rm}{\textnormal{Rm}}

\newcommand{\Ric}{\textnormal{Ric}}

\newcommand{\rii}{\rightarrow\infty}
\newcommand{\ri}{\rightarrow}

\newcommand{\sy}{\mathbb{Z}_2\times O(2)}
\newcommand{\cigar}{\textnormal{Cigar}}

\numberwithin{equation}{section}

\newtheorem{theorem}{Theorem}[section]
\newtheorem{lem}[theorem]{Lemma}

\newtheorem{prop}[theorem]{Proposition}

\theoremstyle{definition}
\newtheorem{defn}[theorem]{Definition}
\newtheorem*{theorem*}{Theorem}

\usepackage{xpatch}
\makeatletter
\xpatchcmd{\tableofcontents}{\contentsname \@mkboth}{\small\contentsname \@mkboth}{}{}
\xpatchcmd{\listoffigures}{\chapter *{\listfigurename }}{\chapter *{\small\listfigurename }}{}{}
\makeatother

\makeatletter
\def\blfootnote{\xdef\@thefnmark{}\@footnotetext}
\makeatother

\begin{document}

\begin{abstract}
For every $\theta\in(0,\pi)$, we construct a 3D steady gradient Ricci soliton whose asymptotic cone is a sector with angle $\theta$, which is a called 3D flying wing.
\end{abstract}

\blfootnote{The author was supported by the NSF grant DMS-2203310.}

\title[3D flying wings for any asymptotic cones]{3D flying wings for any asymptotic cones}

\author[Yi Lai]{Yi Lai}
\email{yilai@stanford.edu}
\address[]{Department of Mathematics, Stanford University, CA 94305, USA}

\maketitle

\section{Introduction}
Ricci solitons are self-similar solutions of the Ricci flow equation. They play an important role in the study of Ricci flow and they often arise as singularity models.
In particular, a steady gradient Ricci soliton is a smooth complete Riemannian manifold $(M,g)$ together with a smooth function $f$ on $M$ which satisfy
\begin{equation}
    \Ric=\nabla^2 f.
\end{equation}
where $f$ is called a potential function.
The soliton generates a Ricci flow for all time by $g(t)=\phi_{-t}^*(g)$, where $\{\phi_t\}_{t\in(-\infty,\infty)}$ is the one-parameter group of diffeomorphisms generated by $\nabla f$ with $\phi_0$ the identity.

In dimension 2, the only non-flat steady gradient Ricci soliton is Hamilton's cigar soliton, which is rotationally symmetric \cite{cigar}.
In dimension $3$, the only non-flat rotationally symmetric steady gradient Ricci soliton is the Bryant soliton \cite{bryant}.
Another example is the product of $\R$ and the cigar soliton.
Note that the asymptotic cones of the Bryant soliton and $\R\times\cigar$ are a ray and a half-plane respectively.

Therefore, Hamilton conjectured that there exist 3d steady gradient Ricci solitons whose asymptotic cones are sectors with angles $\theta\in(0,\pi)$, which are called flying wings \cite{CaoHD,infinitesimal,HaRF,Catino,Chow2007a,DZ}.
In \cite{Lai2020_flying_wing}, the author confirmed this conjecture by constructing a family of $\sy$-symmetric 3D flying wings, which are parametrized by the ratio of Ricci curvature eigenvalues at a point.
However, it remains unknown whether the asymptotic cone angles of these flying wings can take all values in $(0,\pi)$.

Our first theorem gives a positive answer to this question, and thus completes the discussion of the existence of 3D flying wings.
Throughout the paper, the quadruple $(M,g,f,p)$ denotes a steady gradient Ricci soliton, where $f$ is the potential function and $p$ is a critical point of $f$.
\begin{theorem}[Existence]\label{t: flying wing with prescribed angles}
For any $\theta\in(0,\pi)$, there exists a $\mathbb{Z}_2\times O(2)$-symmetric 3D flying wing $(M,g,f,p)$ which is asymptotic to a sector with angle $\theta$.
\end{theorem}

With Theorem \ref{t: flying wing with prescribed angles} proved now, the classification of all 3D steady gradient Ricci solitons remains to prove the uniqueness. More precisely, it remains to see whether the soliton is unique for each asymptotic cone angle $\theta\in[0,\pi]$. For $\theta=0$, the uniqueness of the Bryant soliton as the only one asymptotic to a ray was confirmed very recently by the author in \cite[Theorem 1.1]{Lai2022_O(2)}. For $\theta=\pi$, it is easy to see that the soliton must be isometric to $\R\times\cigar$.
So the uniqueness question is reduced to that of the 3D flying wings for each $\theta\in(0,\pi)$, which all satisfy the $O(2)$-symmetry due to the author \cite[Theorem 1.2]{Lai2022_O(2)}.

We also remark that
in the mean curvature flow, 
the analogs of 3D steady gradient Ricci solitons on $\R^3$ are convex translators in $\R^3$, where the term flying wing denotes translators that are graphs over finite slabs in $\R^2$.
It has been proved that for each $\theta\in(0,\pi)$, there exists a unique mean curvature flow flying wings in $\R^3$ asymptotic to $\theta$ \cite{Wangxujia,Spruck2020CompleteTS,white,langford1}.

Let $\mathcal{M}$ be the space of all $\mathbb{Z}_2\times O(2)$-symmetric 3D steady gradient Ricci solitons on $\R^3$ that are pointed at the critical point where the scalar curvature $R$ is equal to $1$, and consider the map $\tau:\mathcal{M}\ri[0,\pi]$ from the solitons to their asymptotic cone angles.
Then
Theorem \ref{t: flying wing with prescribed angles} shows that $\tau$ is surjective. Moreover, our second theorem shows that the space $\mathcal{M}$ is subsequentially compact under the smooth topology, and the map $\tau$ is continuous.
\begin{theorem}[Compactness]\label{t: theorem compactness}
Let $\{(M_i,g_i,f_i,p_i)\}_{i=1}^{\infty}$ be a sequence of $\sy$-symmetric 3D steady gradient Ricci solitons with asymptotic cone angles $\alpha_i\in[0,\pi]$, and $R(p_i)=1$.
Then there exists a subsequence smoothly converging to a 3D steady gradient Ricci soliton.

Moreover,
suppose $\lim_{i\rii}\alpha_i=\alpha$. Then any subsequential limit of $(M_i,g_i,f_i,p_i)$ has asymptotic cone angle equal to $\alpha$.
In particular, the limit is isometric to the Bryant soliton when $\alpha=0$, and to $\R\times\cigar$ when $\alpha=\pi$.
\end{theorem}

Note that if we assume the uniqueness of 3D flying wings for each asymptotic cone angle, then Theorem \ref{t: theorem compactness} will also imply the continuity of the inverse of the map $\tau$, and hence the continuity of the parametrization of 3D steady gradient Ricci solitons by their asymptotic cone angles.
Theorem \ref{t: flying wing with prescribed angles} and \ref{t: theorem compactness} can also be generalized to higher dimensional $O(n-2)\times O(2)$-symmetric flying wings for any $n\ge3$, see \cite{Lai2021Thesis} for the definition and construction.

The main ideas to prove Theorem \ref{t: flying wing with prescribed angles} are the following: By the asymptotic uniqueness theorem in \cite{Lai2020_flying_wing}, we know that the asymptotic cone angles are uniquely determined by the supremum of $R$ at infinity.
So we will construct a sequence of 3D flying wings with $R=R_0$ at a sequence of points going to infinity, for some fixed $R_0>0$. Then we will show that $\sup R$ on $M\setminus B_g(p,r)$ barely decreases for all sufficiently large $r$ in any 3D flying wings $(M,g,f,p)$. So the sequence of flying wings will converge to a flying wing in which $\limsup_{x\rii} R=R_0$. 

To analyze the asymptotic behavior of $R$, we need a dimension reduction theorem, which says that the geometry looks sufficiently like $\R\times\cigar$, if $R\ge R_0$ and the points are far enough away from $p$. 
Since $R$ and the warping function in $\R\times\cigar$ determine each other, we can reduce the asymptotic analysis of $R$ to that of the warping function, which can then be studied by the distance distortion estimates developed by the author in \cite{Lai2022_O(2)}.
A key ingredient to prove the uniform dimension reduction is the existence of a critical point in all 3D steady gradient Ricci solitons, which was proved in \cite[Theorem 1.4]{Lai2022_O(2)}.

This paper was written during the author’s visit at Beijing International Center for Mathematical Research in summer 2022, and she thanks Gang Tian and Xiaohua Zhu for conversations and their hospitality.
She also wants to thank Richard Bamler for comments.

\section{Preliminaries}
In this section, we introduce some useful notions and recall several results from \cite{Lai2020_flying_wing} and \cite{Lai2022_O(2)}.

First, to measure the closeness from the pointed manifolds to the smooth limits, we adopt the following notion of $\epsilon$-isometry and $\epsilon$-closeness.

\begin{defn}[$\epsilon$-isometry between manifolds]\label{d: epsilon isometry}
Let $\epsilon>0$ and $m\in\mathbb{N}$.
Let $(M^n_i,g_i)$, $i=1,2$, be an n-dimensional Riemannian manifolds, $x_i\in M_i$. We say a smooth map $\phi:B_{g_1}(x_1,\epsilon^{-1})\ri M_2$, $\phi(x_1)=x_2$, is an $\epsilon$-isometry in the $C^m$-norm if it is a diffeomorphism onto the image, and
\begin{equation*}
    |\nabla^{k}(\phi^*g_2-g_1)|\le \epsilon\quad \textit{on}\quad B_{g_1}(x_1,\epsilon^{-1}),\quad k=0,1,...,m,
\end{equation*}
where the covariant derivatives and norms are taken with respect to $g_1$.
We also say $(M_2,g_2,x_2)$ is $\epsilon$-close to $(M_1,g_1,x_1)$ in the $C^m$-norm. 
In particular,  if $m=[\epsilon^{-1}]$, then we simply say $(M_2,g_2,x_2)$ is $\epsilon$-close to $(M_1,g_1,x_1)$ and $\phi$ is an $\epsilon$-isometry. 
\end{defn}

In a non-negatively curved complete non-compact Riemannian manifold, we can equip a length metric on the space of geodesic rays. Moreover, a blow-down sequence of this manifold converges to the metric cone over the space of rays in the Gromov-Hausdorff sense, which we call the asymptotic cone of the manifold, see e.g. \cite[Prop 5.31]{MT}.

We know that 3D steady gradient Ricci solitons have non-negative sectional curvature, and a 3D steady gradient Ricci soliton is isometric to quotients of $\R\times\cigar$ if the curvature is not strictly positive everywhere. 
If the curvature is strictly positive, then the soliton is diffeomorphic to $\R^3$. 
By a simple adaptation of \cite[Lemma 4.2]{Lai2020_flying_wing} we see that the asymptotic cone of any 3D steady gradient Ricci soliton is isometric to a metric cone over an interval $[0,\alpha]$, for some $\alpha\in[0,\pi]$.

In \cite{Lai2022_O(2)} the author showed that all 3D steady gradient Ricci solitons $(M,g)$ are $O(2)$-symmetric. So there is a complete unit speed geodesic $\Gamma:(-\infty,\infty)\ri M$ fixed by the $O(2)$-action, $\Gamma(0)=p$, such that the metric on $M\setminus\Gamma$ is a warped product metric with $S^1$-fibers over a 2D totally geodesic submanifold.
Moreover, there is a quantitative relation between the asymptotic cone angle $\alpha$ and the limit of $R$ along $\Gamma$.

\begin{lem}[Asymptotic Uniqueness]
(\cite[Theorem 1.6]{Lai2022_O(2)})
\label{l: quantitative relation}
Let $(M,g,f,p)$ be a 3d steady gradient Ricci soliton on $\R^3$, whose asymptotic cone is a metric cone over the interval $[-\frac{\alpha}{2},\frac{\alpha}{2}]$ for some $\alpha\in[0,\pi]$.
Let $\Gamma:(-\infty,\infty)\ri M$ be the complete geodesic fixed by the $O(2)$-isometry, then
\begin{equation}
    \lim_{s\rii}R(\Gamma(s))=R(p)\sin^2\frac{\alpha}{2}.
\end{equation}
\end{lem}

It is clear that in the soliton $\R\times\cigar$, the potential function $f$ can be chosen as the direct sum of a linear function in the direction of $\R$ and the potential function on $\cigar$. So there is not a critical point of $f$ when the linear function is not a constant.
The following lemma shows that the critical point of the potential function always exists in all the other 3D steady gradient Ricci solitons. Moreover, by the soliton identity 
\begin{equation*}
    R+|\nabla f|^2=\textnormal{const.},
\end{equation*}
it is easy to see that the critical point is also a maximum point of $R$.
Moreover, the critical point is unique when the curvature is positive.

\begin{lem}(cf. \cite[Theorem 1.6]{Lai2022_O(2)})\label{l: max}
Let $(M,g,f)$ be a 3D steady gradient Ricci soliton with positive curvature, then there exists a point $p\in M$ which is a critical point of the potential function $f$. 
\end{lem}

In \cite[Lemma 4.3]{Lai2020_flying_wing} and \cite[Theorem 1.7]{Lai2022_O(2)}, we proved that the scalar curvature decays quadratically in distance to $\Gamma$ in a 3D steady gradient soliton on $\R^3$. In the following lemma, we show that this curvature estimate is uniform for all 3D steady gradient Ricci solitons on $\R^3$. 

\begin{lem}[Quadratic curvature decay]\label{l: curvature upper bound initial}
There exists a constant $C>0$ such that the following holds:
Let $(M,g)$ be a 3D steady gradient Ricci soliton on $\R^3$.
Then for any $x\in M\setminus\Gamma$, we have
\begin{equation*}
 R(x)\le\frac{C}{d_g^2(x,\Gamma)}.
\end{equation*}
\end{lem}

\begin{proof}
By Perelman's curvature estimate for non-collapsed Ricci flows with non-negative curvature operators \cite[Corollary 11.6]{Pel1}, it suffices to find a constant $\kappa>0$ (independent of $(M,g)$) such that for all $x\in M\setminus\Gamma$ in the universal covering $(\widetilde{B_g(x,d_g(x,\Gamma))},\widetilde{g},\widetilde{x})$ of $(B_g(x,d_g(x,\Gamma)),g,x)$, we have
\begin{equation}\label{e: kappa}
    vol(B_{\widetilde{g}}(\widetilde{x},d_g(x,\Gamma)))\ge\kappa\, d^3_g(x,\Gamma).
\end{equation}

To see this, let $y\in\Gamma$ be a point such that $d_g(x,y)=d_g(x,\Gamma)$.
Then take $z=\Gamma(s)$ for some sufficiently large $s$ so that $d_g(y,z)\ge100\,d_g(x,y)$,
and hence $d_g(x,z)\ge 99\,d_g(x,y)$ by triangle inequality. This implies
\begin{equation}\label{e: small angle}
    \widetilde{\measuredangle}xzy\le 0.1.
\end{equation}
Moreover, let $\sigma_1,\sigma_2:[0,1]\ri M$ be minimizing geodesics from $y$ to $x$, and $y$ to $z$ respectively. 
By the first variation formula, $\sigma'_1(0)$ is orthogonal to $\Gamma$ at $y$. Therefore, by the $O(2)$-symmetry of the soliton we may replace $\sigma_2$ by its image under a suitable isometry in the $O(2)$-action, and assume $\measuredangle(\sigma'_1(0),\sigma'_2(0))\le\frac{\pi}{2}$.
By angle comparison this implies $\widetilde{\measuredangle}xyz\le\frac{\pi}{2}$,
which combining with \eqref{e: small angle} implies
\begin{equation}\label{e: pi20.1}
    \widetilde{\measuredangle}yxz\ge\frac{\pi}{2}-0.1.
\end{equation}

Let $N$ be a totally geodesic surface in $M$ so that the metric on $M\setminus\Gamma$ can be written as $g=g_N+\varphi^2d\theta^2$, where $\varphi$ is a positive smooth function on $N$.
We may assume $x\in N$ and the two minimizing geodesics $xy$ and $xz$ are both contained in $N$.
Let $y'$ and $z'$ be two points on $xy$ and $xz$ such that $d_g(x,y')=d_g(x,z')=\frac{1}{2}d_g(x,y)$.
Then by the angle monotonicity and \eqref{e: pi20.1} we obtain
\begin{equation*}
    \widetilde{\measuredangle}y'xz'\ge\widetilde{\measuredangle}yxz\ge\frac{\pi}{2}-0.1,
\end{equation*}
and hence 
\begin{equation*}
    |\partial B_N(x,\tfrac{1}{2}d_g(x,y))|\ge d_g(y',z')\ge C^{-1}d_g(x,y).
\end{equation*}
So by volume comparison on $N$ this implies
\begin{equation}\label{e: Nvolume}
    vol(B_N(x,\tfrac{1}{2}d_g(x,y)))\ge C^{-1}\,d^2_g(x,y).
\end{equation}
Then as in \cite[Lemma 4.3]{Lai2020_flying_wing}, we can show that \eqref{e: Nvolume} implies \eqref{e: kappa}, which hence proves the lemma.

\end{proof}

\section{Compactness of 3D steady gradient Ricci solitons}
In this section, we prove several compactness results of 3D steady gradient Ricci solitons, and study the asymptotic behavior of the scalar curvature at infinity along $\Gamma$.
Since the subset $\Gamma\setminus\{p\}$ is a union of two integral curves of $\nabla f$, it follows by the soliton identity $\langle\nabla R,\nabla f\rangle=-\Ric(\nabla f,\nabla f)$ that
$R(\Gamma(s))$ decreases in $s$ on $[0,\infty)$, and increases on $(-\infty,0]$. 
The main result in this section is Proposition \ref{l: R does not change too much}, which shows that $R$ barely decreases along $\Gamma$ starting from $\Gamma(s_0)$, if $R(\Gamma(s_0))$ has a lower bound and $s_0$ is sufficiently large. This is the key ingredient in the proofs of Theorem \ref{t: flying wing with prescribed angles} and \ref{t: theorem compactness}.

We will also assume that $dr^2+g_c$ is the product metric on $\R\times\cigar$ such that $R(r,x_{tip})=1$, $r\in\R,x_{tip}\in\cigar$, and $g_{stan}$ is the flat product metric on $\RR\times S^1$ such that the length of each $S^1$-factor is equal to $4\pi$. Note that for any sequence of points $q_i\in \cigar$, $q_i\rii$, the manifolds $(\cigar,g_c,q_i)$ smoothly converge to $\R\times S^1$ where the length of the $S^1$-factor is equal to $4\pi$.

\begin{lem}\label{compactness to a steady soliton}
Let $(M_i,g_i,f_i,p_i)$ be a sequence of 3D steady gradient Ricci solitons on $\R^3$ satisfying $R(p_i)=1$.
Suppose there exist $\epsilon>0$ and $q_i=\Gamma_i(s_i)\in M_i$ for some $s_i\in\R$ such that $R(q_i)\ge\epsilon$ for all $i$.
Then after passing to a subsequence, the solitons $(M_i,g_i,f_i-f_i(q_i),q_i)$ smoothly converge to a non-flat 3D steady gradient Ricci soliton $(M_{\infty},g_{\infty},f_{\infty},q_{\infty})$ on $\R^3$, and $q_{\infty}\in\Gamma_{\infty}$, where $\Gamma_{\infty}$ is the unit speed complete geodesic fixed by the $O(2)$-isometry and $f_{\infty}(q_{\infty})=0$.
\end{lem}

\begin{proof}
First, since $(M_i,g_i)$ is non-compact complete, the curvature is positive, and $R_{max}=R(p_i)=1$, by a well-known fact of Gromoll and Meyer (see \cite{CheegerEbin}), we always have an injectivity radius lower bound
\begin{equation}\label{e: v_0}
    \inf_{x\in M_i}\textnormal{inj}_{g_i}(x)\ge\frac{\pi}{\sqrt{R_{\max}}}=\pi.
\end{equation}
Therefore, by Hamilton's compactness of Ricci flow \cite{Hamilton_compactness} we may pass to a subsequence and assume that the Ricci flows $(M_i,g_i(t),q_i)$, $t\in(-\infty,0]$, smoothly converge to a Ricci flow $(M_{\infty},g_{\infty}(t),q_{\infty})$, $t\in(-\infty,0]$.

Next, we show that  $(M_{\infty},g_{\infty}(t),q_{\infty})$ is the Ricci flow of a steady gradient Ricci soliton.
To see this, let $\widetilde{f}_i=f_i-f_i(q_i)$, then $\widetilde{f}_i(q_i)=0$. Moreover, by $|\nabla \widetilde{f}_i|=|\nabla f_i|\le 1$ and the soliton equation
\begin{equation*}
    \nabla^2 \widetilde{f}_i=\nabla^2 f_i=\Ric_{g_i},
\end{equation*}
we can deduce that for any integer $k\ge1$, there is $C_k>0$ such that $|\nabla^k \widetilde{f}_i|\le C_k$.
Therefore, after passing to a subsequence we may assume the functions $\widetilde{f}_i$ on $M_i$ smoothly
converge to a smooth function $f_{\infty}$ on $M_{\infty}$, which satisfies $\Ric_{g_{\infty}(0)}=\nabla^2 f_{\infty}$.
So $(M_{\infty},g_{\infty}(0),q_{\infty})$ is a 3D steady gradient Ricci soliton which satisfies
\begin{equation*}
    R(q_{\infty})=\lim_{i\rii}R(q_i)\ge\epsilon>0.
\end{equation*}

It remains to show that $M_{\infty}$ is diffeomorphic to $\R^3$. For this, it suffices to show that $(M_{\infty},g_{\infty}(0),q_{\infty})$ is not isometric to a recaling of $S^1\times\cigar$.
Suppose it is isometric to a recaling of $S^1\times\cigar$, then
let $\psi_i:(S^1\times\cigar,g_{\infty}(0),q_{\infty})\ri (M_i,g_i,q_i)$ be an $\epsilon_i$-isometry, where $\epsilon_i\ri0$ as $i\rii$.
Let $V=S^1\times\{x_{tip}\}\subset S^1\times\cigar$, and $B_r(V)=\{x\in S^1\times\cigar: d_{g_{\infty}(0)}(x,V)< r\}$ for any $r>0$.
We claim that $\Gamma_i(-\infty,\infty)\subset \psi_i(B_1(V))$.
To see this, note that the $O(2)$-isometry on $(M_i,g_i)$ converge to the $O(2)$-isometry on $S^1\times\cigar$ that fixes $S^1\times{x_{tip}}$. Letting $X_i$ and $X_{\infty}$ be the corresponding killing fields on $M_i$ and $S^1\times\cigar$, then it follows that $(\psi^{-1}_{i})_*(X_i)$ smoothly converge to $X_{\infty}$ as $i\rii$.
First, we have $q_i=\Gamma_i(s_i)=\psi_i(q_{\infty})\in \psi_i(B_1(V))$.
Next, suppose $\Gamma_i(s)\in \psi_i(B_1(V))$ for some $s\in\R$, then by $X_i(\Gamma_i(s))=0$ we see that $\Gamma_i(s)\in\psi_i(B_{\frac{1}{10}}(V))$, and hence $\Gamma_i([s-\frac{1}{2},s+\frac{1}{2}])\subset B_{g_i}(\Gamma_i(s),\frac{1}{2})\subset \psi_i(B_1(V))$. Therefore, by induction we can deduce $\Gamma_i(-\infty,\infty)\subset \psi_i(B_1(V))$, which contradicts the non-compactness of $\Gamma_i$. 
\end{proof}

Next, we show a special case of the compactness Lemma \ref{compactness to a steady soliton}, where the limit is $\R\times\cigar$.
For a fixed 3D flying wing, we know that it converges to a rescaling of $\R\times\cigar$ along $\Gamma$, see \cite{Lai2020_flying_wing,Lai2022_O(2)}. 
Lemma \ref{l: |nabla f| small} and \ref{l: looks like RxCigar} provide sufficient conditions, under which the geometry along $\Gamma$ is close to $\R\times\cigar$.
These conditions are uniform for all 3D steady gradient Ricci solitons.
First, Lemma \ref{l: |nabla f| small} shows that the closeness to $\R\times\cigar$ is guaranteed when $|\nabla f|(\Gamma(s))$ is sufficiently small and $s$ is sufficiently large.

\begin{lem}\label{l: |nabla f| small}
For any $\epsilon>0$, there exist $D(\epsilon),\delta(\epsilon)>0$ such that the following holds:

Let $(M,g,f,p)$ be a 3D steady gradient Ricci soliton on $\R^3$, and $R(p)=1$. 
Suppose $|\nabla f|(\Gamma(s))\le\delta(\epsilon)$ for some $|s|>D(\epsilon)$, then the pointed manifold $(M,g,\Gamma(s))$ is $\epsilon$-close to $(\R\times\cigar,(0,x_{tip}))$.
\end{lem}

\begin{proof}
Without loss of generality we may assume $s>0$.
Suppose this is not true for some $\epsilon>0$, then we can find sequence of numbers $D_i\rii$ and $\delta_i\ri0$ and a sequence of 3D steady gradient Ricci solitons $(M_i,g_i,p_i)$ with $R(p_i)=1$, such that $|\nabla f|(\Gamma_i(D_i))\le \delta_i$, but the pointed manifolds $(M_i,g_i,\Gamma_i(D_i))$ is not $\epsilon$-close to $\R\times\cigar$.
Since for any $s\ge0$ we have
\begin{equation}
    \frac{d^2}{ds^2}f(\Gamma(s))=\nabla^2 f(\Gamma'(s),\Gamma'(s))=\Ric(\Gamma'(s),\Gamma'(s))\ge 0,
\end{equation}
it follows that $|\nabla f|(\Gamma(s))$ decreases in $s$, and thus $|\nabla f|(\Gamma_i(s))\le\delta_i$ for all $s\in[0,D_i]$.

By Lemma \ref{compactness to a steady soliton} we may assume the manifolds $(M_i,g_i,f_i-f_i(\Gamma_i(D_i)),\Gamma_i(D_i))$ converge to a steady gradient Ricci soliton $(M_{\infty},g_{\infty},f_{\infty},p_{\infty})$ on $\R^3$.
Let $\Gamma_{\infty}:(-\infty,\infty)\ri M$ be the complete geodesic fixed by the $O(2)$-isometry, and $\Gamma_{\infty}(0)=p_{\infty}$, then it is easy to see that $|\nabla f_{\infty}|(\Gamma_{\infty}(s))\equiv0$ for all $s\le0$, and $R(p_{\infty})=1$.
In particular, this implies $\Ric(\Gamma'_{\infty}(0),\Gamma'_{\infty}(0))=\nabla^2 f_{\infty}(\Gamma'_{\infty}(0),\Gamma'_{\infty}(0))=0$, and hence $(M_{\infty},g_{\infty},p_{\infty})$ is isometric to $\R\times\cigar$, which is a contradiction for large $i$.
\end{proof}

We say an $n$-dimensional Riemannian manifold $M$ dimension reduces to an $(n-1)$-dimensional manifold $N$ along a sequence of points $x_i\in M$, if the manifolds $(M,R(x_i)g,x_i)$ smoothly converge to $\R\times N$.

We know that a 3D flying wing always dimension reduces to $\cigar$. However, the $\epsilon$-closeness to $\R\times\cigar$ for a fixed $\epsilon>0$ may happen at any arbitrarily large distance to the critical point, because the soliton may be very close to the Bryant soliton. 
In the following we prove a dimension reduction which is uniform for all 3D flying wings.
It shows that there is a uniform distance for the $\epsilon$-closeness to be achieved, as long as $R$ is uniformly bounded away from zero. A key ingredient in the proof is Lemma \ref{l: max} (the existence of the critical point).

\begin{lem}[Uniform dimension reduction]\label{l: looks like RxCigar}
For any $R_0\in(0,1],\epsilon>0$, there exists $D(R_0,\epsilon)>0$ such that the following holds:

Let $(M,g,f,p)$ be a 3D steady gradient Ricci soliton on $\R^3$.
Suppose $R(p)=1$. 
Then for any $|s|>D(R_0,\epsilon)$, if $R(\Gamma(s))\ge R_0$, then the manifold $(M, R(\Gamma(s))g,\Gamma(s))$ is $\epsilon$-close to $(\R\times\cigar,(0,x_{tip}))$.
\end{lem}

\begin{proof}
Suppose the lemma is false for some $\epsilon>0$, then we can find a sequence of $(M_i,g_i,f_i,p_i)$ which are 3D steady gradient Ricci solitons on $\R^3$, and a sequence of numbers $D_i\rii$ such that $R(p_i)=1$, $R(\Gamma_i(D_i))\ge R_0$, but $(M_i, R(\Gamma_i(D_i))g_i,\Gamma_i(D_i))$ is not $\epsilon$-close to $(\R\times\cigar,(0,x_{tip}))$ for each $i$.
First, since $R_0>0$, by Lemma \ref{compactness to a steady soliton} we may assume after passing to a subsequence that $(M_i, R(\Gamma_i(D_i))g_i,\Gamma_i(D_i))$ smoothly converge to a steady gradient Ricci soliton $(M_{\infty},g_{\infty},f_{\infty},p_{\infty})$ on $\R^3$, and $R(p_{\infty})=1$. 

First, if $|\nabla f_i|(\Gamma_i(D_i))\ri0$, then we get a contradiction by Lemma \ref{l: |nabla f| small}.
So we may assume $\limsup_{i\rii}|\nabla f_i|(\Gamma_i(D_i))>0$. Therefore, by the soliton identity $R+|\nabla f|^2=1$ and passing to a subsequence we may assume 
\begin{equation}\label{e: Rinfinity}
    R_0\le\lim_{i\rii}R(\Gamma_i(D_i))<1.
\end{equation}
Then we divide the discussion into two cases depending on whether there is a critical point of $f_{\infty}$ on the unit speed complete geodesic $\Gamma_{\infty}:(-\infty,\infty)\ri M_{\infty}$ fixed by the $O(2)$-isometry, and $\Gamma_{\infty}(0)=p_{\infty}$.

\textbf{Case 1:} Suppose there exists $s_0\in\R$ such that $|\nabla f_{\infty}|(\Gamma_{\infty}(s_0))=0$.
Then it follows that $|\nabla f_i|(\Gamma_i(D_i+s_0))\ri0$ as $i\rii$.
Since $D_i+s_0\rii$, by Lemma \ref{compactness to a steady soliton} we see that $(M_i,g_i,\Gamma_i(D_i+s_0))$ smoothly converge to $(\R\times\cigar,(0,x_{tip}))$, and the geodesics $\widetilde{\Gamma}_i(s):=\Gamma_i(s+D_i+s_0)$, $s\in\R$, in $(M_i,g_i)$ smoothly converge to the line $\R\times\{x_{tip}\}$ in $\R\times\cigar$ modulo the diffeomorphisms.
In particular, this implies
\begin{equation*}
    \lim_{i\rii}R(\Gamma_i(D_i))=\lim_{i\rii}R(\widetilde{\Gamma}_i(-s_0))=1,
\end{equation*}
which contradicts with the assumption \eqref{e: Rinfinity}.

\textbf{Case 2:} Suppose $|\nabla f_{\infty}|(\Gamma_{\infty}(s))>0$ for all $s\in\R$. We claim that $(M_{\infty},g_{\infty})$ must be isometric to $\R\times\cigar$. Suppose not, then it has positive curvature, and by Lemma \ref{l: max} there exists a unique critical point $q$ of $f_{\infty}$, which is also the unique maximum point of $R$. Therefore, $q$ is fixed by the $O(2)$-isometry and hence must be on $\Gamma_{\infty}$, which contradicts our assumption.
\end{proof}

Assume $f(p)=0$ where $p$ is the critical point. Let $\Sigma=f^{-1}(s)$ be a level set of $f$ for some $s>0$. 
Then $\Sigma$ is a compact, $O(2)$-symmetric, smooth 2D Riemannian manifold. Moreover, since the second fundamental form satisfies $\mathrm{I\!I}=-\left.\frac{\nabla ^2 f}{|\nabla f|}\right|_{\Sigma}=-\left.\frac{\Ric}{|\nabla f|}\right|_{\Sigma}\le 0$,
it follows by the Gauss equation that $\Sigma$ has positive Gaussian curvature.
Moreover, $\Sigma$ intersects $\Gamma$ at two different points.
In the following lemma, we obtain a lower bound on $R$ at the two points, in terms of the warping function at certain points in $\Sigma$.

\begin{lem}\label{l: geometry of level set near tip}
For any $D>0$, there exists $C(D)>0$ such that the following holds:

Let $(M,g,f,p)$ be a 3D steady gradient Ricci soliton with positive curvature.
Suppose $R(p)=1$. For any point $x\in M$ with $\varphi(x)\ge D^{-1}$, let $y$ be one of the two intersection points of $\Gamma$ and the level set $f^{-1}(f(x))$ that is closest to $x$ with respect to the induced metric on $\Sigma$.
Suppose $d_g(x,\Gamma)\ge4\,\varphi(x)$.
Then we have 
\begin{equation*}
    R(y)\ge C^{-1}\, \varphi^{-2}(x).
\end{equation*}
\end{lem}

\begin{proof}
Suppose this is false, then we can find a sequence of 3D steady gradient Ricci solitons  $(M_i,g_i,f_i,p_i)$ on $\R^3$, with $R(p_i)=1$, and $x_i,y_i\in M_i$ satisfying the assumptions, but
\begin{equation}\label{e: R and varphi}
    R(y_i)\varphi_i^{2}(x_i)\ri0\quad\textit{as}\quad i\rii.
\end{equation}
By $\varphi_i(x_i)\ge D^{-1}$ we have $R(y_i)\ri0$ as $i\rii$. So we may assume $|\nabla f_i|\ge\frac{1}{2}$ for all $i$.
Consider the rescaled metrics $\widetilde{g}_i=\varphi_i^{-2}(x_i)g$ and the rescaled functions $\widetilde{f}_i:=\frac{f_i-f_i(y_i)}{\varphi_i(x_i)}$, then $\widetilde{f}_i$ satisfies $\widetilde{f}_i(y_i)=0$ and
\begin{equation}\label{e: higher derivatives of f}
    \widetilde{\nabla}^2\widetilde{f}_i=\nabla^2\widetilde{f}_i=\frac{\nabla^2 f_i}{\varphi_i(x_i)}=\frac{\Ric}{\varphi_i(x_i)}=\frac{\widetilde{\Ric}}{\varphi_i(x_i)},
\end{equation}
and also $|\widetilde{\nabla}\widetilde{f}_i|_{\widetilde{g}_i}=
    \varphi_i(x_i)|\nabla\widetilde{f}_i|_{g_i}=
    |\nabla f_i|_{g_i}\le 1$.
In particular, at $y_i$ we have
\begin{equation}\label{e: the norm of gradient}
    |\widetilde{\nabla}\widetilde{f}_i|_{\widetilde{g}_i}(y_i)=|\nabla f_i|(y_i)\in\left[\frac{1}{2},1\right].
\end{equation}
By \eqref{e: higher derivatives of f} and \eqref{e: the norm of gradient}, the derivatives of $\widetilde{f}_i$ are uniformly bounded.
Using \eqref{e: R and varphi} we can show by a limiting argument as in \cite[Lemma 3.3]{Lai2020_flying_wing} that
the manifolds $(M_i,\varphi^{-2}_i(x_i)g_i,y_i)$ smoothly converge to the 3D Euclidean space $(\R^3,0)$.

Moreover, a subsequence of the functions $\widetilde{f}_i$ converge to a smooth function $f_{\infty}$ on $\R^3$ with $f_{\infty}(0)=0$. 
By \eqref{e: higher derivatives of f} and \eqref{e: the norm of gradient} it satisfies $\nabla^2f_{\infty}=0$ and $|\nabla f_{\infty}|(0)>0$. 
So $f_{\infty}$ is a non-constant linear function.
In particular, $0$ is a regular value of $f_{\infty}$, so the 
the level sets $(\Sigma_i,R(y_i)g_{\Sigma_i},y_i)$ of $\widetilde{f}_i$ with induced metrics smoothly converge to the level set $(f_{\infty}^{-1}(0),0)$,
which is isometric to the 2D Euclidean space with the induced metric.

Let $\sigma_i:[0,1]\ri\Sigma_i$ be a $\Sigma_i$-minimizing geodesic from $y_i$ to $x_i$.
Since $y_i$ is closer to $x_i$ between the two points in $\Gamma_i\cap \Sigma_i$, by the concavity of $\varphi_i$ on $\Sigma_i$, it is easy to see 
\begin{equation}\label{e: largest circle}
    \sup_{s\in[0,1]}\varphi_i(\sigma_i(s))\le 2\varphi_i(x_i).
\end{equation}
By the assumption $d_{g_i}(x,\Gamma_i)\ge 4\,\varphi_i(x_i)$, we have that
\begin{equation}\label{e: distance of the circle is not too small compare to the circle length}
    d_{\Sigma_i}(y_i,x_i)\ge d_{g_i}(y_i,x_i)\ge d_{g_i}(x_i,\Gamma_i) \ge 4\,\varphi_i(x_i).
\end{equation}
Write the induced metric on $f^{-1}_{\infty}(0)$ in the warped product form $dr^2+\varphi_{\infty}^2(r)d\theta^2$ so that $r=0$ at $0\in f^{-1}_{\infty}(0)$.
Then $\varphi_{\infty}(r)=r$ because $f^{-1}_{\infty}(0)$ is isometric to $\R^2$.
But \eqref{e: largest circle} and \eqref{e: distance of the circle is not too small compare to the circle length} imply that $\varphi_{\infty}(r)\le2$ for all $r\le4$, which is a contradiction.
\end{proof}

In some of the following results,
we will moreover assume that the soliton is $\mathbb{Z}_2$-symmetric for simplicity, by which we mean that there is a $\mathbb{Z}_2$-isometry $\tau$ on $M$, which fix the critical point $p$ and the differential map $\tau_{*p}$ is a reflection in $T_pM$ that maps the vector $\Gamma'(0)$ to $-\Gamma'(0)$.

The next lemma shows that if $R\ge R_0$ at some point on $\Gamma$ that is sufficiently far away from the critical point. Then $R$ has a uniform lower bound at the infinity of $\Gamma$, which only depends only on $R_0$.
Lemma \ref{l: looks like RxCigar} is needed in the proof: First,
it allows us to reduce the estimate of $R$ to that of the warping function.
Second, it implies the initial condition needed to apply the ODE estimate for the warping function. So we can obtain an upper bound on the warping function, which in turn implies a lower bound on $R$.

\begin{lem}\label{l: a rough lower bound C}
For any $R_0\in(0,1)>0$, there exist $D(R_0),C(R_0)>0$ such that the following holds:

Let $(M,g,f,p)$ be a $\sy$-symmetric 3D steady gradient Ricci soliton with positive curvature.
Suppose $R(p)=1$. 
Suppose also that there is $s_0>D(R_0,\epsilon)$ such that $R(\Gamma(s_0))\ge R_0$.
Then for all $|s|\ge s_0$,
\begin{equation*}
    R(\Gamma(s))\ge C^{-1}>0.
\end{equation*}
\end{lem}

\begin{proof}
We can write the metric $g$ on $M\setminus\Gamma$ as $g=g_N+\varphi^2d\theta$ where $N$ is a totally geodesic surface.
Let $N_0$ be the 2D submanifold fixed by the $\mathbb{Z}_2$-isometry, then $M\setminus N_0$ has two connected components $N_+,N_-$. It is not hard to see that for any point $x\in N_+$ (or $N_-$), the point $\phi_t(x)\in N_+$ (or $N_-$) for all $t\in\R$. Without loss of generality we may assume $\Gamma_+=\Gamma(0,\infty)\subset N_+$. So $d_g(x,\Gamma)=d_g(x,\Gamma_+)$ when $x\in N_+$.

In the following $C$ denotes all positive constants depending only on $R_0$, whose values may change from line to line.
Let $C_0>10$ be a constant whose value will be determined later. Let $\epsilon>0$ be sufficiently small.
By Lemma \ref{l: looks like RxCigar} we can find constant $D>0$ depending on $R_0$ and $\epsilon$, so that $(M,R(\Gamma(s_0))g,\Gamma(s_0))$ is $\epsilon$-close to $\R\times\cigar$. In particular, we can find a point $x$ such that $d_g(x,\Gamma)=d_g(x,\Gamma_+)\ge C_0\,\varphi(x)$
and also 
\begin{equation}\label{e: h(0)}
  1\le 1.9\,R_0^{-1/2}\le \varphi(x)\le 2.1\,R_0^{-1/2}.
\end{equation}
Let $H(t)=d_g(\phi_t(x),\Gamma)$ and $h(t)=\varphi(\phi_t(x))$. Then we have
\begin{equation}
    \frac{H(0)}{h(0)}=\frac{d_g(x,\Gamma)}{\varphi(x)}\ge C_0.
\end{equation}
Let
\begin{equation*}
    a=\sup\{t\ge0: H(t)\ge4\,h(t)\}\in(0,\infty].
\end{equation*}
We will show $a=\infty$. 

First, since $x\notin\Gamma$, it follows that $\phi_t(x)\notin\Gamma$. In particular, $f(\phi_t(x))>0$ and the level set $\Sigma_t=f^{-1}(f(\phi_t(x)))$ for each $t\ge0$ is a compact 2D submanifold. Assume $\Sigma_t$ intersects with $\Gamma_+$ at $y_t$. 
Since $\phi_t(x)\in N_+$, there is a point $q_t\in\Gamma_+$ such that $d_g(q_t,\phi_t(x))=d_g(\phi_t(x),\Gamma)$. Then by \cite[Lemma 3.29]{Lai2022_O(2)} we see that $f(q_t)\le f(\phi_t(x))=f(y_t)$.
The conditions $d_{\Sigma_t}(\phi_t(x),\Gamma\cap\Sigma_t)=d_{\Sigma_t}(\phi_t(x),y_t)$, $\varphi(\phi_t(x))\ge\varphi(x)\ge 1$, and $H(t)\ge4h(t)$ allow us to apply
Lemma \ref{l: geometry of level set near tip} and deduce $R(y_t)\ge C^{-1}h^{-2}(t)$. Then by the monotonicity of $R$ along $\Gamma_+$, we see that the following holds in $[0,a]$,
\begin{equation}\label{e: R(q)}
    R(q_t)\ge R(y_t)\ge C^{-1}h^{-2}(t).
\end{equation}

For any point $\Gamma(s)\in\Gamma$, let $e_1,e_2\in T_{\Gamma(s)}M$ be two unit vectors which together with $e_3:=\Gamma'(s)$ form an orthogonal basis.
Then by the $O(2)$-symmetry we have $R=4K(e_1,e_2)+2K(e_1,e_3)$. Together with $\Ric(e_1,e_1)=K(e_1,e_2)+K(e_1,e_3)$ this implies
\begin{equation}\label{e: quarter}
    \Ric(e_1,e_1)=\Ric(e_2,e_2)\ge\frac{1}{4}R.
\end{equation}
Therefore, 
let $\gamma_t:[0,H(t)]\ri M$ be a minimizing geodesic connecting $q_t$ and $\phi_t(x)$, then by \eqref{e: quarter}, \eqref{e: R(q)}, and the uniform bound on the curvature derivatives we have
\begin{equation*}
    \Ric(\gamma_t'(r),\gamma_t'(r))\ge C^{-1}\cdot h^{-2}(t)
\end{equation*}
for all $r\in[0,C^{-1}h(t)]$.
Since $H(t)\ge4\,h(t)$, it follows that
\begin{equation*}
    H'(t)=\int_0^{H(t)}\Ric(\gamma_t'(r),\gamma_t'(r))\,dr\ge C^{-1}\cdot h^{-1}(t).
\end{equation*}
So there are constants $C_1,C_2>0$ that only depend on $R_0$ such that the following inequalities hold for all $t\in[0,a]$,
\begin{equation}\label{e: ODE derivative assump}
    \begin{cases}
     H'(t)\ge C_1^{-1}\cdot h^{-1}(t)\\
     h'(t)\le C_2\cdot H^{-2}(t)\cdot h(t),
    \end{cases}
\end{equation}
where the second inequality is a consequence of the Ricci flow equation and Lemma \ref{l: curvature upper bound initial}.

We may assume $\frac{2C_1C_2}{e}\ge5$ and take $C_0=2C_1C_2$, then it follows by the ODE estimates \cite[Lemma 3.37]{Lai2022_O(2)} that
\begin{equation}\label{e: caseHh}
    \begin{cases}
    H(t)\ge C_3t+H(0)\\
    h(t)\le h(0)e^{\frac{C_2}{C_3H(0)}},
    \end{cases}
\end{equation}
for all $t\in[0,a]$, where $C_3=C_1^{-1} h^{-1}(0)-C_2H^{-1}(0)>0$.
So $h(t)\le h(0)\,e$, and hence $\frac{H(t)}{h(t)}\ge\frac{1}{e}\frac{H(0)}{h(0)}\ge5$.
By the supremum of $a$ this implies $a=\infty$. 
So by \eqref{e: R(q)} and \eqref{e: h(0)} we obtain $R(q_t)\ge C^{-1}$
for all $t\ge0$. Note that by \cite[Lemma 3.19]{Lai2022_O(2)} we have that $q_t\rii$ as $t\rii$.
So this implies
$\lim_{s\rii}R(\Gamma(s))\ge C^{-1}$, and thus proves the lemma.

\end{proof}

The next lemma shows that the distance between any two points that are not on $\Gamma$ will stay bounded under the backwards Ricci flow, see also \cite[Theorem 3.39]{Lai2022_O(2)}.

\begin{lem}\label{l: two points stay bounded}
Let $(M,g,f,p)$ be a 3d steady gradient Ricci soliton on $\R^3$. Then for any $x_1,x_2\in M\setminus\Gamma$, there exists $C>0$ (which may depend on $x_1,x_2$ and $(M,g)$) such that $d_{g(t)}(x_1,x_2)=d_g(\phi_t(x_1),\phi_t(x_2))<C$ for all $t\ge0$.
\end{lem}

\begin{proof}
Let $C>0$ denote all constants whose values may change from line to line.
First, by $\Rm\ge0$ we see that $d_g(\phi_t(x_j),\Gamma)$ increases in $t$, for $j=1,2$. Moreover, by \cite[Theorem 1.5]{Lai2022_O(2)} it is not hard to see that
\begin{equation}\label{e: grows linearly}
    d_g(\phi_t(x_j),\Gamma)\ge d_g(x_j,\Gamma)+C^{-1}t\ge C^{-1}(t+1).
\end{equation}
So by Lemma \ref{l: curvature upper bound initial} (quadratic curvature decay) we see that $R(x)\le\frac{C}{(t+1)^2}$ holds for all $x\in B_g(x_j,C^{-1}(t+1))$, for each $j=1,2$. So by Perelman's distance distortion estimate \cite[8.3(b)]{Pel1} we have
\begin{equation*}
    \frac{d}{dt}d_g(\phi_t(x_1),\phi_t(x_2))\le\frac{C}{t+1},
\end{equation*}
integrating which we obtain
\begin{equation}\label{e: lnt}
    d_g(\phi_t(x_1),\phi_t(x_2))\le d_g(x_1,x_2)+C\ln (t+1).
\end{equation}
Therefore, for any sufficiently large $t$, let $\gamma_t:[0,1]\ri M$ be a minimizing geodesic between $\phi_t(x_1),\phi_t(x_2)$, by \eqref{e: grows linearly} and the triangle inequality we have
\begin{equation*}
    d_g(\gamma_t([0,1],\Gamma)\ge d_g(\phi_t(x_1),\Gamma)-d_g(\phi_t(x_1),\phi_t(x_2))> C^{-1}(t+1).
\end{equation*}
So by Lemma \ref{l: curvature upper bound initial} we have $\sup_{s\in[0,1]}R(\gamma_t(s))\le \frac{C}{(t+1)^2}$, and hence \eqref{e: lnt} implies
\begin{equation*}
    \frac{d}{dt}d_g(\phi_t(x_1),\phi_t(x_2))\le\int_{\gamma}\Ric(\gamma_t'(s),\gamma_t'(s))\,ds\le \frac{C\ln (t+1)}{(t+1)^2}\le \frac{C}{(t+1)^{\frac{3}{2}}},
\end{equation*}
integrating which we proved the lemma.
\end{proof}

Lastly, we prove the main result in this section, which gives a condition for $R$ to be stable along $\Gamma$.
More precisely, it says that if $R(\Gamma(s_0))\ge R_0$ for some sufficiently large $s_0$ depending on $R_0$, then the value of $R(\Gamma(s))$ barely drops on $s\in [s_0,\infty)$. So $\lim_{s\rii}R(\Gamma(s))$ is sufficiently close to $R_0$.
The proof relies on Lemma \ref{l: looks like RxCigar} which allows us to convert the comparison of $R$ to that of the warping functions.

\begin{prop}\label{l: R does not change too much}
For any $R_{\#}\in(0,1],\epsilon>0$, there exists $D(R_{\#},\epsilon)>0$ such that the following holds:

Let $(M,g,f,p)$ be a $\sy$-symmetric 3D steady gradient Ricci soliton with positive curvature.
Suppose $R(p)=1$. 
Suppose also that there is $s_0>D(R_{\#},\epsilon)$ such that $R(\Gamma(s_0))=R_0\ge R_{\#}$.
Then for all $s\in\R$, $|s|\ge s_0$, we have
\begin{equation*}
    R_0(1-\epsilon)\le R(\Gamma(s))\le R_0.
\end{equation*}
\end{prop}

\begin{proof}
Let $\delta>0$ be a constant that we shall take arbitrarily small, and $\epsilon>0$ be all constants so that $\epsilon\ri0$ as $\delta\ri0$. Let $D,C>0$ denote all constants depending on $\R_{\#}$ and $\delta$.

Let $R_{\infty}:=\lim_{s\rii}R(\Gamma(s))=\lim_{s\ri-\infty}R(\Gamma(s))$.
First, by Lemma \ref{l: a rough lower bound C} we see that $R_{\infty}\ge C^{-1}>0$.
So by Lemma \ref{l: looks like RxCigar} we may assume $D$ to be sufficiently large so that for any $|s|>D$, the manifold $(M, R(\Gamma(s))g,\Gamma(s))$ is $\delta$-close to $(\R\times\cigar,(0,x_{tip}))$.
So we can find two points $x_1,x_2\in M$ such that $d_g(x_1,\Gamma),d_g(x_2,\Gamma)\ge\epsilon^{-1}$, and
\begin{equation}\label{e: h(x) and R}
    |\varphi(x_1)-2(R_0)^{-1/2}|\le\epsilon,\quad\textit{and}\quad |\varphi(x_2)-2(R_{\infty})^{-1/2}|\le\epsilon.
\end{equation}

Next, by using $R_{\infty}\ge C^{-1}$ and \eqref{e: quarter} we can deduce $\frac{d}{dt}d_g(\phi_t(x_j),\Gamma)\ge C^{-1}$,
integrating which we have $d_g(\phi_t(x_j),\Gamma)\ge d_g(x_j,\Gamma)+C^{-1}t$, for $j=1,2$.
Combining this with Lemma \ref{l: curvature upper bound initial} (quadratic curvature decay), we obtain
\begin{equation}\label{e: R(phi_t(x_j))}
    R(\phi_t(x_j))\le\frac{C}{d_g^2(\phi_t(x_j),\Gamma)}\le\frac{C}{(C^{-1}t+d_g(x_j,\Gamma))^2}\le\frac{C}{(C^{-1}t+\epsilon^{-1})^2}.
\end{equation}
Since $2\pi\cdot\varphi(\phi_t(x))$ is equal to the $g$-length of the $S^1$-orbit at $\phi_t(x)$, which is equal to the $g(-t)$-length of the $S^1$-orbit at $x$, it hence follows by the Ricci flow equation and $\Rm\ge0$ that
\begin{equation*}
    0\le\frac{d}{dt}\varphi(\phi_t(x_j))\le C\,R(\phi_t(x_j))\,\varphi(\phi_t(x_j)),
\end{equation*}
integrating which and using \eqref{e: R(phi_t(x_j))} we obtain
\begin{equation}\label{e: compare h(x) and h(phi_t(x))}
    \varphi(x_j)\le\varphi(\phi_t(x_j))\le(1+\epsilon)\varphi(x_j).
\end{equation}

Since $d_g(\phi_t(x),\Gamma)\rii$, by \cite[Theorem 1.5]{Lai2022_O(2)} we see that 
the manifold is $\epsilon$-close to $\RR\times S^1$ at $\phi_t(x_1)$ for all sufficiently large $t$. So by Lemma \ref{l: two points stay bounded} it is easy to see
\begin{equation*}
    (1-\epsilon)\varphi(\phi_t(x_2))\le \varphi(\phi_t(x_1))\le(1+\epsilon)\varphi(\phi_t(x_2)).
\end{equation*}
Combining this with \eqref{e: compare h(x) and h(phi_t(x))} and \eqref{e: h(x) and R} we obtain
\begin{equation}\label{e: suffices}
    R_0(1-\epsilon)\le R_{\infty}\le R_0(1+\epsilon),
\end{equation}
and hence proves the lemma.
\end{proof}

\section{Proof of main results}

In this section we prove Theorem \ref{t: flying wing with prescribed angles} and \ref{t: theorem compactness}.

\begin{proof}[Proof of Theorem \ref{t: flying wing with prescribed angles}]
First, as in the proof of \cite[Theorem 1.1]{Lai2020_flying_wing}, 
we can find a sequence of smooth families of $\mathbb{Z}_2\times O(2)$-symmetric expanding gradient Ricci solitons $(M_{i,\mu},g_{i,\mu},p_{i,\mu}),\mu\in[0,1]$, $i\in\mathbb{N}$, with positive curvature operator \cite{De15}, which satisfies the following conditions,
\begin{enumerate}
    \item $R(p_{i,\mu})=1$ for all $i\in\mathbb{N}$ and $\mu\in\R$;
    \item $(M_{i,0},g_{i,0},p_{i,0})$ are rotationally symmetric for all $i$, and $(M_{i,0},g_{i,0},p_{i,0})$ smoothly converge to the Bryant soliton as $i\rii$;
    \item $(M_{i,1},g_{i,1},p_{i,1})$ smoothly converge to $\R\times\cigar$  as $i\rii$.
    \item For any sequence $\mu_i\in[0,1]$, a subsequence of $(M_{i,\mu_i},g_{i,\mu_i},p_{i,\mu_i})$ smoothly converges to a $\sy$-symmetric 3D steady gradient Ricci soliton on $\R^3$.
\end{enumerate}
By abuse of notation, we will use $\Gamma$ to denote the unit speed complete geodesic in any expanding gradient soliton $(M_{i,\mu},g_{i,\mu},p_{i,\mu})$ that is fixed by the $O(2)$-isometry.

For any $\theta\in(0,\pi)$, let $R_0=\sin^2\frac{\theta}{2}\in(0,1)$, we now construct a 3D flying wing $(M_{\infty},g_{\infty},p_{\infty})$ such that $R(p_{\infty})=1$ and $\lim_{s\rii}R(\Gamma(s))=\lim_{s\ri-\infty}R(\Gamma(s))=R_0$.
First,
by Proposition \ref{l: R does not change too much} we can choose a sequence of numbers $\{s_k\}_{k=0}^{\infty}$ so that $s_k\rii$ as $k\rii$, and if $R(\Gamma(s_j))\ge R_0$ holds in a $\sy$-symmetric 3D steady gradient Ricci soliton on $\R^3$, then the following will hold for all $s\ge s_j$,
\begin{equation}\label{e: choice of s_j}
    R(\Gamma(s))\ge (1+(j+1)^{-1})^{-1}R(\Gamma(s_j)).
\end{equation}

Since $\lim_{s\rii}R(\Gamma(s))=0$ in the Bryant soliton,
we may take $s_k$ to be sufficiently large so that $R(\Gamma(s_k))<R_0$ in the Bryant soliton. We also see that $R(\Gamma(s_k))=1>R_0$ in $\R\times\cigar$. So by condition (2)(3), we can find a $\mu_{i,k}\in(0,1)$ for each fixed $k$ and all sufficiently large $i$ so that $R_{g_{i,\mu_{i,k}}}(\Gamma(s_k))=R_0$.
By condition (1)(4), for each fixed $k$, we may assume after passing to a subsequence that the expanding gradient Ricci solitons $(M_{i,\mu_{i,k}},g_{i,\mu_{i,k}},p_{i,\mu_{i,k}})$ smoothly converge to a $\mathbb{Z}_2\times O(2)$-symmetric 3D flying wing
$(M_k,g_k,p_k)$, which satisfies $R_{g_{k}}(p_k)=1$ and $R_{g_{k}}(\Gamma(s_k))=R_0$.
So by the monotonicity of $R$ along $\Gamma$ we have $R_{g_k}(\Gamma(s))\ge R_0$, and hence by \eqref{e: choice of s_j} we obtain that for each $j=1,...,k$, and for all $s_{j-1}\le s\le s_k$,
\begin{equation*}
    R_0\le R_{g_k}(\Gamma(s))\le R_{g_k}(\Gamma(s_{j-1}))\le(1+j^{-1})R_{g_k}(\Gamma(s_k))=(1+j^{-1})R_0.
\end{equation*}

By Lemma \ref{compactness to a steady soliton} we may assume after passing to a subsequence that the 3D flying wings $(M_k,g_k,p_k)$ smoothly converge to a $\mathbb{Z}_2\times O(2)$-symmetric 3D flying wing $(M_{\infty},g_{\infty},p_{\infty})$, which satisfies $R_{g_{\infty}}(p_{\infty})=1$ and the following holds for all $j\in\mathbb{N}_+$,
\begin{equation*}
    R_0\le R_{g_{\infty}}(\Gamma(s))\le(1+j^{-1})R_0, \quad\textit{for all}\quad s\ge s_{j-1}.
\end{equation*}
In particular, this implies $\lim_{s\rii}R_{g_{\infty}}(\Gamma(s))=\lim_{s\ri-\infty}R_{g_{\infty}}(\Gamma(s))=R_0$.
So by Lemma \ref{l: quantitative relation} we see that $(M_{\infty},g_{\infty})$ is asymptotic to a sector with angle $\theta$.   
\end{proof}

Now we prove Theorem \ref{t: theorem compactness}.

\begin{proof}[Proof of Theorem \ref{t: theorem compactness}]
Let $(M_i,g_i,f_i,p_i)$ be a sequence of 3D steady gradient Ricci solitons whose asymptotic cone angles are $\alpha_i$ and $\lim_{i\rii}\alpha_i=\alpha$.
Then by Lemma \ref{compactness to a steady soliton}, any converging subsequence of $(M_i,g_i,f_i,p_i)$ converges to a 3D steady gradient Ricci soliton $(M,g,f,p)$ on $\R^3$.

First, assume $\alpha=0$. Suppose by contradiction that the asymptotic cone angle of $(M,g,f,p)$ is equal to some $\beta>0$, which by Lemma \ref{l: quantitative relation} implies $\lim_{s\rii}R(\Gamma(s))=\sin^2\frac{\beta}{2}>0$. So for any $s_0>0$, we have
\begin{equation*}
    \lim_{i\rii}R_{g_i}(\Gamma_i(s_0))=R(\Gamma(s_0))\ge\sin^2\frac{\beta}{2}.
\end{equation*}
So by Proposition \ref{l: R does not change too much} there exists $C>0$ such that  $\lim_{s\rii}R_{g_i}(\Gamma_i(s))\ge C^{-1}$ holds for all sufficiently large $i$. So by Lemma \ref{l: quantitative relation} we have $\liminf_{i\rii}\alpha_i>0$, which contradicts the assumption $\alpha=0$.
Therefore, we have $\beta=0$.
Moreover, by \cite[Theorem 1.1]{Lai2022_O(2)}, it follows that $(M,g,f,p)$ is isometric to the Bryant soliton.

So we may assume $\alpha>0$. Then by Lemma \ref{l: quantitative relation} we have $\lim_{i\rii}\lim_{s\rii}R_{g_i}(\Gamma_i(s))=\sin^2\frac{\alpha}{2}$.
Therefore, by applying Proposition \ref{l: R does not change too much} in each $(M_i,g_i,f_i,p_i)$ we see that for any $\epsilon>0$, there exists $s_0>0$ and $N\in\mathbb{N}$ such that for all $s\ge s_0$ and $i\ge N$, we have
\begin{equation*}
    \left|R_{g_i}(\Gamma_i(s))-\sin^2\frac{\alpha}{2}\right|\le\epsilon.
\end{equation*}
Passing this to the limit we obtain
\begin{equation*}
    |R(\Gamma(s))-\sin^2\frac{\alpha}{2}|\le\epsilon
\end{equation*}
for all $s\ge s_0$ in $(M,g,f,p)$. Letting $\epsilon\ri0$, we get $\lim_{s\rii}R(\Gamma(s))=\sin^2\frac{\alpha}{2}$, which proves the theorem by Lemma \ref{l: quantitative relation}.

\end{proof}

\bibliography{bib}
\bibliographystyle{abbrv}

\end{document}